\documentclass[12pt,a4paper,twoside]{amsart}

\usepackage{amsmath}

\usepackage{amssymb}

\usepackage{amsfonts}

\usepackage{a4}

\newtheorem{theorem}{Theorem}[section]

\newtheorem{lemma}[theorem]{Lemma}

\newtheorem{proposition}[theorem]{Proposition}

\newtheorem{corollary}[theorem]{Corollary}

\theoremstyle{definition}

\newtheorem{definition}[theorem]{Definition}

\newtheorem{claim}[theorem]{Claim}

\newtheorem{example}[theorem]{Example}

\newtheorem{question}[theorem]{Question}

\newtheorem{conjecture}[theorem]{Conjecture}

\newtheorem{remark}[theorem]{Remark}

\newcommand{\mA}{\mathbb A}

\newcommand{\mC}{{\mathbb C}}

\newcommand{\mE}{{\mathbb E}}

\newcommand{\mF}{\mathbb F}

\newcommand{\mL}{\mathbb L}

\newcommand{\mM}{\mathbb M}

\newcommand{\mR}{{\mathbb R}}

\newcommand{\mT}{\mathbb T}

\newcommand{\mV}{\mathbb V}

\newcommand{\mW}{\mathbb W}

\newcommand{\mX}{\mathbb X}

\newcommand{\mY}{\mathbb Y}

\newcommand{\mZ}{{\mathbb Z}}

\newcommand{\ho}{\hookrightarrow}

\newcommand{\Gg}{\gamma}

\newcommand{\GG}{\Gamma}

\newcommand{\bo}{\omega}

\newcommand{\bl}{\lambda}

\newcommand{\bL}{\Lambda}

\newcommand{\bs}{\sigma}

\newcommand{\D}{\Delta}

\newcommand{\kk}{\kappa}

\newcommand{\mcC}{\mathcal C}

\newcommand{\mcF}{\mathcal F}

\newcommand{\mcG}{\mathcal G}

\newcommand{\mcI}{\mathcal I}

\newcommand{\mcL}{\mathcal L}

\newcommand{\mcP}{\mathcal P}

\newcommand{\mcU}{\mathcal U}

\newcommand{\ti}{\tilde}

\newcommand{\emp}{\emptyset}

\usepackage{color}
\usepackage{soul}

\begin{document}
\title{Extending weakly polynomial functions from high rank varieties}

\begin{abstract}Let $k$ be a field, $V$ a $k$-vector space and $X$ be a  subset of $V $. 
A function  $f:X\to k$ is {\em weakly polynomial} of degree $\leq a$, if  the  restriction of $f$ on any  affine subspace $L\subset X$ is a  polynomial of degree $\leq a$. In this paper we consider the case when 
$X= \mX (k)$ where $\mX$ is a complete intersection of bounded codimension  defined by a  high rank  polynomials of degrees $d, char(k)=0$ or 
$char (k)>d$ and either $k$ is algebraically closed, or $k=\mF _q,q>ad$. We show that under these assumptions any $k$-valued weakly polynomial function of degree $ \leq a$  on $X$ is a restriction of a polynomial  of degree $\leq a$  on $V$.

Our proof is based on Theorem \ref{Jan} on fibers of polynomial morphisms
$P:\mF _q^n\to \mF _q^m$ of high rank. This result is of an independent interest. For example it immediately implies a  strengthening of the  result  of \cite{Jan}.
 \end{abstract}

\author{David Kazhdan}
\address{Einstein Institute of Mathematics,
Edmond J. Safra Campus, Givaat Ram 
The Hebrew University of Jerusalem,
Jerusalem, 91904, Israel}
\email{david.kazhdan@mail.huji.ac.il}

\author{Tamar Ziegler}
\address{Einstein Institute of Mathematics,
Edmond J. Safra Campus, Givaat Ram 
The Hebrew University of Jerusalem,
Jerusalem, 91904, Israel}
\email{tamarz@math.huji.ac.il}

\thanks{The second author is supported by ERC grant ErgComNum 682150. 
}

\maketitle

\section{introduction} 
Let $k$ be a field. We denote $k$-algebraic varieties by bold letters such as $\mX$ and the sets of $k$-points of $\mX$ by  $\mX (k)$  or by $X$. We fix $d,a,c\geq 1$. We always assume that $|k|>ad$, that there exists a root of unity $\beta \in k$ of order $m>2a$,  and that either char$(k)>d$ or that it is of characteristic $0$. A field is  {\em admissible}
if it satisfying these conditions. 

\begin{definition}\label{weak-def}
Let  $V$ be a $k$-vector space  and  $X\subset V$. We say that a function $f:X \to k$ is {\it weakly polynomial} of degree $\leq a$ if the restriction $f_{|L}$ to any affine subspace  $L \subset X$ is a polynomial of degree $\leq a$.  
\end{definition}

\begin{remark} If $|k|>a$ it suffices to check this on $2$-dimensional subspaces (see \cite{KR}).
Namely a function is {\it weakly polynomial} of degree $\leq a$ if the restriction $f_{|L}$ to $2$-dimensional affine subspace  $L \subset X$ is a polynomial of degree $\leq a$.
\end{remark}

One of  goals of this paper is to construct classes of hypersurfaces $\mX \subset \mV$ such  that
any weakly polynomial function $f$ on $X$ of degree $\leq a$ is a restriction of a polynomial $F$ of degree $\leq a$ on $V$. The main difficulty is in the case when  $a\geq d$ since in this case an extension $F$ of $f$ to $V$ is not unique.

To state our result properly we introduce some definitions:

\begin{definition}Let $\mX \subset \mV$ be an algebraic $k$-subvariety.
\begin{enumerate}
\item $\mX \subset \mV$ satisfies $\star ^k_{a}$
 if any weakly polynomial function of degree $\leq a$ on $X$ is a restriction  of a polynomial function of degree $\leq   a$ on $V$.
\item A $k$-subvariety $\mX \subset \mV$ satisfies $\star _{a}$ if $\mX (l)$ satisfies $\star^l_a$ for 
 any finite extension $l/k$.
 \item For any family $\bar P=\{ P_i\}, 1\leq i\leq c$ of polynomials on $\mV$ we define the subscheme $\mX _{\bar P}\subset \mV$ by the system of equations $\{P_i(v)=0\}, 1\leq i\leq c $.
\end{enumerate}
\end{definition}

The following example demonstrates the  existence of cubic 
surfaces $\mX \subset \mA^2$  which do not have the property 
$\star^k_{1}$ for any field $k$.

\begin{example} Let  $V=k^2$, $Q=xy(x-y)$. Then $X=X_0\cup X_1\cup X_2$ where 
$X_0=\{ v\in V|x=0\}, X_1=\{ v\in V|y=0\} , X_2=\{ v\in V|x=y\} $.
The function $f:X\to k$ such that $f(x,0)=f(0,y)=0,f(x,x)=x$ is weakly linear but one can not extend it to a linear function on $V$.
\end{example}

Our goal is to prove that high rank hypersurfaces have the property $\star _a^k$. 

\begin{definition}[Algebraic rank\footnote{Also known as the Schmidt $h$-invariant.}]\label{def-rank} 
Let $P:V \to k$ be a  polynomial.  We define the $k$-rank $r_k(P)$ as the minimal number $r$ such that  $P=\sum_{i=1}^r Q_iR_i$, with $Q_i, R_i$ of $k$-polynomials of degrees strictly smaller than the degree of $P$. When $k$ is fixed we write $r(P)$ instead of $r_k(P)$.
\end{definition}

\begin{question}  Does there exists a function $R(d,r)\to \infty$ for $r\to \infty$ such that $r_{\bar k}(P)\geq R(d,r) $ for any 
polynomial $P$ of degree $d$ and $k$-rank $r$, where $\bar k$ is the algebraic closure of $k$?
\end{question}

Our main result is that high rank hypersurfaces over admissible fields satisfy   $\star^k_a$.

\begin{theorem}\label{main} 
There exists  $r=r(a,d)$ such that for any admissible  field $k$, any $k$-vector space $\mV$, and any  hypersurface $\mX \subset \mV$ of degree $d$ and rank $\geq r$ satisfies  $\star^k_a$.
\end{theorem}

\begin{conjecture}  There exists  $r=r(a,d)$ such that for any admissible  field $k$, any $k$-vector space $\mV$,  any hypersurface $\mX \subset \mV$ of degree $d$, and rank $\ge r$, the associated  hypersurface $\ti \mX \subset \mV^d$  satisfies  $\ti \star^k_a$.
\end{conjecture}

The first step in our proof of Theorem \ref{main} is to construct  an  explicit collection of hypersurfaces  $\mX_n$ of  rank $\sim n$
satisfying  $\star^k_a$ for admissible fields $k$.

Let $\mW=\mA ^d$, let $\mV_n=\mW ^n$, and let
$P_n:\mV_n \to \mA$ be given by $P_n(w_1,\dots ,w_n)= \sum _{i=1}^n \mu(w_i)$, where  $\mu:\mW \to \mA$  is given  by $\mu(x^1, \dots ,x^d):= \prod _{j=1}^dx^j$.

\begin{theorem}\label{const} 
\begin{enumerate}\item $r(P_n)\geq n/d$.
\item 
For any admissible field $k$,  the hypersurface  $\mX _n\subset \mV _n$ has the property  $\star _a$.
\end{enumerate}
\end{theorem}

The second step in our proof of Theorem \ref{main} is of an independent interest. To formulate the statement we introduce some definitions.
\begin{definition}\label{proj}
Let $\mV$ be $k$-vector space.
\begin{enumerate}
\item We denote by $\mcP _d(\mV)$ the $k$-variety of polynomials of degree $\leq d$ on $V$.
\item For any $k$-vector space $\mW$ we denote by $Hom_{af}(\mW, \mV)$ the $k-$variety of affine maps $\phi :\mW \to \mV$.
\item For any $P\in \mcP _d(V) $ we denote 
by $\kk _P: Hom_{af}(W,V) \to \mcP _d(\mW) $ the map given by $\phi \to \phi ^\star (P) \ (:=P\circ \phi)$.
\item A field $k$ has the property $c_d$ 
if for any $m \geq 2$, there exists $\rho =\rho (m,d)$ such that for any $k$-vector spaces $\mV,\mW, \ \dim(\mW) \leq m$, a polynomial   $P:\mV \to \mA$ of degree $d$ and rank $\geq \rho $ the map $\kk _P(k) : Hom_{af}(W,V) \to \mcP _d(\mW) $ is onto.
\item A field $k$ has the property $c_\infty$ if it has the property $c_d$ for all $d$.
\end{enumerate}
\end{definition}

\begin{theorem}\label{Jan}
  \begin{enumerate}
\item  Finite fields of characteristic $>d$ have the property $c_\infty$.
\item For any algebraically closed fields of either characteristic $>d$ or characteristic $0$ and  $Q\in  \mcP _d(W)(k) $ the fiber $\kk ^{-1}_P(Q) $ is a variety of dimension $dim (Hom_{af}(W,V))-dim(\mcP _d(\mW))$.
\end{enumerate}
\end{theorem}

\begin{remark}
 Theorem \ref{Jan} is first proven for finite fields. We show in Appendix $C$ how to derive the validity Theorem \ref{Jan} for algebraically closed fields  from the corresponding statement  for finite fields.
\end{remark}

The results extend without difficulty to complete intersections $\mX \subset \mV$ of bounded degree and codimension, and high rank (see Definition \ref{def-rank-variety})
\begin{theorem}\label{main1} 
For any  $c>0$, there exists  $r=r(a,d,c)$  such that for any admissible  field $k$, any $k$-vector space $\mV$, and any  subvariety  $\mX \subset \mV$ codimention $c$, degree $d$ and rank $\geq r$ the subset $\mX \subset \mV$ satisfies  $\star ^k_a$. 
\end{theorem}
\begin{conjecture}
\begin{enumerate}
\item For any $m,d \geq 0$, there exists $\rho =\rho (m,d)$ such that if the rank of a polynomial $P$ is $\geq  \rho (m,d)$, then the map $\kk _P$ is flat.
\item Non-archimedian local fields have the property $c_\infty$.
\item The bound on $r$ depends polynomially on $c$. This will follow from Conjecture \ref{conj-bias}
which is currently known for $d=2,3$ \cite{s-h}. 
\end{enumerate}
\end{conjecture}

\begin{remark}
From now on the claims in the introduction and in  the first three  sections of the main body of  the paper are stated for the case that $\mX$ a hypersurface ($c=1$).  The general case ( which is completely analogous) is discussed in Section $5$.
\end{remark}

The third  step in the proof of Theorem \ref{main} consists of   the following two results. 
Let $\mX \subset \mV$ be a hypersurface of degree $d$, $l:V\to k$ a  non-constant affine function, $\mX _b:=l^{-1}(b)\cap \mX$. We denote by $\mZ$ the variety of affine $2$-dimensional subspaces $\mL \subset \mX _1$ and by $\mZ ^0\subset \mZ$ the constructible subset of planes $\mL$
 for which there exists an affine $3$-dimensional subspace $\mM \subset \mX$ containing $\mL$ and such that $\mM \cap \mX _0\neq \emp$. Let $\mY := \mZ - \mZ ^0$. We denote by $r(\mX)$  the rank of the polynomial defining the hypersurface $\mX _0\subset \mV _n^0:=l^{-1}(0) $. \\

 The first result states that the subvariety  $\mY \subset \mZ$ is {\it small} if $r(\mX)>>1$.
\begin{claim}\label{planes} For any $s>0$, there exists $r=r(d,s)$ such that if $r(\mX) >r$ then
\begin{enumerate}
\item If $k$ is finite and char$(k)>d$ then $|\mY (k)| / |\mcL (k)| \leq  |k|^{-s}$.
\item $dim (\mY)<dim (\mZ)-s$.
\end{enumerate}
\end{claim}

\begin{remark} We first prove the part $(1)$ using  techniques from Additive Combinatorics. The part $(2)$ is then an easy corollary.
\end{remark}

The second result demonstrates that one can extend   any weakly polynomial function vanishing on $X_0$ to   a  polynomial function on $X$:

\begin{proposition}\label{ext}  
There exists $r=r(d,a)$ such that if $r(X) >r$  and $k$ is an 
admissible field then
\begin{enumerate}
\item If $k=\mF _q,q>a$ then any weakly polynomial function on $X$ of degree $\leq a$ vanishing on $X_0$ is a restriction of a polynomial on $V$ of degree $\leq a$.
\item If $k$ is an algebraically closed field then any weakly polynomial function on $X $ 
of degree $\leq a$ vanishing on $X_0$ is a 
restriction of a polynomial on $V$ of degree
 $\leq a$.
\end{enumerate}
\end{proposition}

As an immediate corollary we obtain:
\begin{corollary} \label{extension1}  There exists $r=r(d,a)$ such that 
the following holds for all admissible fields $k$. 

Let  $\mX= \{v\in \mV|P(v)=0\}\subset \mV$ be 
a hypersurface of degree  $\leq d$ and $\mW \subset \mV$ an affine subspace 
such that rank of $P_{|\mW}\geq r$. Then  for any weakly polynomial function $f$ on $X$
of degree $\leq a$ such that $f _{|X\cap W}$  extends to a polynomial on $W$ of degree $\leq a$ there  exists an extension $F$ of $f$ to 
a polynomial on $V$  of degree $\leq a$.
\end{corollary}
\begin{proof} Choose a flag 
$$\mcF =\{W_0=W\subset W_1\dots \subset W_{dim(V)-dim(W)}=V\}, dim (W_i)=dim(W)+i,$$ 
and extend $f$ by induction in $i,1\leq i\leq dim(V)-dim(W) $ to a polynomial $F$ on $V$.
\end{proof}

\begin{remark} The choice of $F$ depends on a choice of flag $\mcF$ and on choices involved in the inductive arguments.
\end{remark}

Theorem \ref{main} is an almost 
immediate corollary of the results stated above:
\begin{proof}[Proof of Theorem \ref{main} assuming Theorems \ref{const} and \ref{Jan},  and Corollary \ref{extension1}]

Let $\ti r$ be from Corollary \ref{extension1} and $r=r(a,d):=\rho (dim (W),d)$  from Theorem \ref{Jan}. As follows from Theorem \ref{const}  hypersurfaces $\mX _n$ are of  rank $\geq \ti r$ for $n\geq d\ti r$.

 Let $\mX \subset \mV$ be a hypersurface of rank $\geq r$.  By Theorem \ref{Jan}  there exists a linear 
map $\phi :\mW \to \mV$ such that $\mX _n=\{ w\in \mW|\phi (w)\in \mX\}$. Since $\mX _n$
satisfies $\star^k_a$,  Corollary \ref{extension1} implies that $\mX$ satisfies $\star ^k_a$. 
\end{proof}

\begin{remark}\label{leq} 
We believe that the restrictions on the characteristic is not necessary in Theorem \ref{main}. 
\end{remark}

\begin{remark} 
 The case $a <d$ was studied in \cite{kz-uniform}. The case $a=d=2$ of was studied in \cite{kz}, and a bilinear version of it was studied in \cite{GM}, where it was applied as part of a  quantitative proof for the inverse theorem for the $U_4$-norms over finite fields. We expect the results in this paper to have similar applications to a quantitative proof for the inverse theorem for the higher Gowers uniformity norms, for which at the moment only a non quantitative proof using ergodic theoretic methods exists \cite{btz, tz, tz-1}. 
\end{remark}

We finish the introduction with a different application of  Theorem \ref{Jan}. We show how to derive the following strengthening of the main Theorem from \cite{Jan}. \\

\begin{lemma} Let $k$ be an algebraically closed field, $\mcC$ the category of finite-dimensional affine $k$-vector spaces with morphisms being affine maps, let   $\mcF _d$ be the contravariant endofunctor  on $\mcC$ given by  
$$\mcF _d(V)=\{\text{Polynomials on $V$ of degree $\leq d$} \},$$
and let  $\mcG \subset \mcF$ be a proper subfunctor. Then there exists $r$ such that $r(P)\leq r$ for any finite-dimensional $k$-vector space $V$ and $P\in \mcG (V)$.
\end{lemma}

\begin{proof} Let $\mcG$ be a subfunctor of $\mcF _d$ such that $r(P), P\in \mcG (W) $ is not bounded above. We want to show that $\mcG (W)=\mcF _d(W)$ for any finite-dimensional $k$-vector space $W$.

Let $m=dim(W)$ and choose a polynomial $P\in \mcG (V)$ where $V$ where $V$ is a $k$-vector space $V$ such that $r(P)\geq r(m,d)$ where  $r(m,d)$ is as in Theorem \ref{Jan}. Then for any polynomial $Q$ on $W$ of degree $d$ there exist an affine map $\phi :W\to V$ such that $Q=\phi^\star (P)$. We see that $\mcG (W)=\mcF _d(W)$.
\end{proof}

\begin{remark} The paper \cite{Jan} assumed that $\mcG (V)\subset \mcP _d(V)$ are Zariski closed subsets.
\end{remark}


\section{Proof of Theorem \ref{const}} Let $\mX _n$ be as in Theorem \ref{const}.
\subsection{Rank of $\mX_n$} In this subsection we prove the part $1$ of 
Theorem \ref{const}. 
\begin{proof} We start with the following general result.

Let $\mV =\mA ^N$, let $P:\mV\to \mA$ be a homogeneous polynomial of degree
 $d$, and let $\mX =\mX _P:=\{ v\in V|P(v)=0\}$. We denote $ \mX _{sing}\subset \mX$ the subvariety of points $x\in \mX$ such that $\frac {\partial P}{\partial x_l}(x)=0,1\leq l\leq N$.
\begin{lemma}\label{cod} $codim _\mX (\mX _{sing})\leq 2 dr(P)$ where $r(P)$ is the $d$-rank of $P$.
\end{lemma}
\begin{proof} By definition we can write $P$ as sum 
$P=\sum _{i=1}^r\ti Q_i\ti R_i$ where $\deg (\ti Q_i)$, $\deg(\ti R_i)<d$, $1\leq d\leq r$, and  $r:=r(P)$. Writing $\ti Q_i,\ti R_i$ as sums of homogeneous components we see that $P=\sum _{j=1}^{r(d-1)} Q_jR_j$ where $\deg (Q_j),\deg( R_j)$ are homogeneous polynomials of degrees $\geq 1$.

Let $\mY =\{ v|P(v)=0, Q_j(v)=0, R_j(v)=0, \}, 1\leq j\leq r(d-1)$. Since $\{ 0\}\in \mY$ we see that $\mY \neq \emp$. Since $\mY \subset \mV$ is defined by   $\leq 2r(d-1)+1$ equations we see that that $codim _\mX(\mY)\leq rd$. 

Since $\mY \subset \mX _{sing} $ we have $codim _\mX (\mX _{sing}) \leq 2 dr. $
\end{proof}

Now we show that $r(P_n)\geq n/d$.
As follows from Lemma \ref{cod} it is sufficient to prove that codim$_\mX (\mX _{sing})= 2n$ where $\mX =\mX _{P_n}$. Let $\mW '\subset \mW$ be the subvariety of points 
$\{x^j\}$ such that $x^a=x^b=0$ for some $a,b,1\leq a<b\leq d$. It is clear that $\mX _{sing} =(\mW ')^n\subset \mW ^n=\mV_n$. Therefore codim$_\mX (\mX _{sing})= n\times \text{codim}_\mW (\mW ')=2n$.
\end{proof} 

\subsection{Proof of the part (2) of Theorem \ref{const}}
\begin{definition}\label{X}
\begin{enumerate} 
\item For any set $X$ we denote by  $k[X]$  the space of $k$-valued functions on $X$.
\item For a subset $X$ of a vector space $V$ we denote by $\mcP _a^w(X) \subset k[X] $ the subspace of 
 weakly polynomial functions of degree $\leq a$.
\item We denote by $\mcP _a(X) \subset \mcP _a^w(X) $ the subspace of functions $f:X\to k$ which are restrictions of polynomial functions on $V$ of degree $\leq a$.
\item  Let $\mW=\mA^d$ and $\mu$ be the map $\mu:\mW \to \mA$ given by 
\[
\mu (a_1,\dots ,a_d)= \prod _{s=1}^d a_s.
\]
\item For any $n$ we define $\mV _n:=\mW ^n$ and write elements 
$v\in V_n$ in the form 
$$v= (w_1,\dots ,w_n), \ 1\leq i\leq n,\ w_i\in W.$$
\item We denote  by $P_n:\mW ^n\to \mA$ the degree $d$ polynomial  $$P_n(w_1,\dots ,w_n):=\sum _{i=1}^n\mu (w_i),$$ and write $\mX _n:=\mX _{P_n} = \{v \in \mV: P_n(v)=0\}$.
\end{enumerate}
\end{definition}

We fix $n$ and write $X$ instead of $X_n$ and $V$ instead of $V_n$. We will use notations from Definition \ref{X}. The proof is based on the  existence of a large group  of symmetries of $X$ and the existence of a linear subspace $L\subset V$ of dimension $dim(V)/d$. Since the field $k$ admissible it contains a fine subgroup $\D \subset k^\star $ of size $m>ad$.

Under this notation Theorem \ref{const} becomes: 
\begin{theorem}\label{equality} Let $k$ be an admissible field, then 
$\mcP _a^w(X) = \mcP _a(X) $.
\end{theorem}

\begin{proof}We start with
the following result.
\begin{claim}\label{many}
Let $Q$ be a polynomial of degree $\leq ad$ on $k^N$ such that $Q_{|\D ^N}\equiv 0.$ Then $Q=0$.
\end{claim}
\begin{proof} The proof is by induction in $N$. If $N=1$   then $Q=Q(x)$ is polynomial such that $Q(\delta) =0$ for $\delta \in \D$. Since $|\D| > ad$ we see that $Q=0$.

Assume that the result is know for $N=s-1$. Let $Q$ be a polynomial of degree $\leq ad$ on $k^s$ such that $Q_{|\D ^s}\equiv 0.$ By induction we see that $Q(\delta ,x_2,\dots ,x_s)\equiv 0$ for all $\delta \in \D$.
Then for any $x_2,\dots ,x_s $ the polynomial $x\to Q(x,x_2,\dots ,x_s)$ vanishes for all $\delta \in \D$. Therefore $Q(x,x_2,\dots ,x_s) =0$ for all $x\in k$.
\end{proof}

\begin{definition}
\begin{enumerate}
\item $\GG :=(S_d)^n$. The group $\GG$ acts naturally on $X$.
\item    $L:=\{(c_1, \ldots, c_n)\in k^n |\sum_{i=1}^nc_i=0\}.$ 
\item $L(m)=(\D) ^n\cap L\subset k^n$.
\item For $c\in k$ we write $w(c):=(c,1,\dots ,1)\in W$.
\item   $\kk :L\ho X\subset L\subset V$ is  the linear map given by 
$$\kk (c_1, \ldots, c_n) := (w(c_1), \ldots, w(c_n)).$$

\item $T_1:= \{ (u_1,\dots ,u_d)\in (\D)^d|\prod _{j=1}^du _j=1\}$. Let $T=T_1^n$.
\item For any $j,j'$, $1\leq j\neq j'\leq d$ we denote by $\phi _{j,j'}: \D \to T_1 $ the 
morphism such that $\phi _{j,j'}(u)= ( x_l(u)), 1\leq l\leq d $ where $x_j(u)=u, x_{j'}(u)=u^{-1} $ 
and $x_l(u) =1$ for $l\neq j,j'$.
\item We denote by $\Theta _1$ the group of homomorphisms  $\chi :T_1\to k^\star$.
\item For $\chi \in \Theta _1 , j,j', 1\leq j\neq j'\leq d $ we define
a homomorphism $\chi _{j,j'}: \D \to k^\star $ by $\chi _{j,j'}:=\chi \circ \phi _{j,j'} $. Since $\D \cong \mZ /m\mZ$ there exists unique $\alpha _{j,j'}(\chi )\in [-m/2,m/2]$ such that 
$\chi _{j,j'}(u)= u^{\alpha _{j,j'} (\chi )}$ for any $u\in \D$. 
\item Denote by $\Theta _1^{adm}=\{ \chi  \in \Theta _1 : | \alpha _{j,j'}(\chi ) |\leq a\}$, and by  $\Theta_1^{adm,+}\subset \Theta_1^{adm} $ the subset of $\chi$ such that $\alpha _{j,j'} \geq 0$ if $j<j'$.
\item Denote by $\Theta=(\Theta _1 )^n$  the groups of characters 
$\theta:T\to k^\star$. 
\item Denote by  $\Theta^{adm,+}:=(\Theta _1 ^{adm,+})^n$ and by
$\Theta^{adm}:=(\Theta _1 ^{adm})^n$.
\item For any $k$-vector space $R$, a representation $\pi : T \to Aut (R)$ and  $\theta \in \Theta $ we define 
$$R^\theta =\{ r\in R|\pi(t)r=\theta (t)r, \ t\in T\}.$$
Since $|T|$ is prime to $q$ we have a direct sum decomposition 
$R=\oplus _{\theta \in \Theta }R^\theta$.
and on $X$.  \\
\item For any function $f:X\to k, \Gg \in \GG$ define a function $h_{\Gg ,f} :L\to  k$ by 
$h_{\Gg ,f} :=f\circ  \kk _\Gg$.
\end{enumerate}
\end{definition}

\begin{claim}\label{pol} For any $f\in \mcP _a^w(X) , \Gg \in \GG $ the function 
$h_{\Gg ,f} $ is a polynomial of degree $\leq a$.
\end{claim}

\begin{proof} Since $f\in \mcP _a^w(X) $ we have  $h_{\Gg ,f}\in \mcP _a^w(L) $. Since $L$ is linear space we see that $h_{\Gg ,f} $ is a polynomial of degree $\leq a$.
\end{proof}

\begin{definition} We denote by $\mcP _a^{\bar w}(X)$ the space of functions $f$ such that $h_{\Gg ,f} $ is a polynomial of degree $\leq a$ on $L$ for all $\Gg \in \GG $.
\end{definition}

\begin{claim}\label{plus}
\begin{enumerate}
\item The subset $\Theta^{adm}$ of $\Theta$ is $\GG$-invariant. 
\item For any $\theta \in \Theta^{adm} $ there exists $\Gg \in \GG$ such that $\theta \circ \Gg \in  \Theta^{adm,+} $
\end{enumerate}
\end{claim}

The group $T$ acts naturally on $X$ and on spaces $\mcP ^{\bar w}_a(X)$ and $\mcP _a(X)$ and we have direct sum decompositions  
\[\mcP ^ {\bar w}_a(X) =\oplus _{\theta \in \Theta}
\mcP ^ {\bar w}_a(X) ^\theta\] 
and 
\[\mcP _a(X) =\oplus _{\theta \in \Theta}
\mcP _a(X) ^\theta.
\]

Therefore  to prove  Proposition \ref{equality} it  suffices to show that $\mcP^{\bar w}_a(X) ^\theta =\mcP _a(X) ^\theta$ for any $\theta \in \Theta$. This will follow from the following statement.

\begin{proposition}\label{Id} 
\begin{enumerate}
\item  For any function $f:X\to k$  satisfying the equation 
$f(tx)=\theta (t)f(x), t\in T ,x\in X, \theta \in \Theta$ and such that $h_{\Gg ,f} $ are polynomial functions on $L$ of degree $\leq a$ for all $\Gg \in \GG$ there exists a polynomial $F$ on $V$ of degree $\leq ad$ such that $f=P_{|X}$.
\item Let $f:X\to k$ be a weakly polynomial of degree $<a$ on $X$ which 
is a restriction of  polynomial function of degree $\leq ad$ on $V$. Then $f$ is a restriction of 
polynomial function of degree $\leq a$ on $V$. 
\end{enumerate}
\end{proposition}

We start  a proof of the first part of Proposition \ref{Id} with a set of notation.
Let $f:X\to k$ satisfy  $f(tx)=\theta (t)f(x), t\in T ,x\in X$ and such that $h_{\Gg ,f} $ are polynomial functions on $L$ of degree $\leq a$ for all $\Gg \in \GG$. 

\begin{definition}
\begin{enumerate}
\item  We write   $h,h_\Gg :L \to k$ instead of $h_{Id,f}$ and $h_{\Gg ,f}$.
\item $W^0:=\{ w= (x_1,\dots ,x_d)|x_i\in \D $ for $i\geq 2\}\subset W$.
\item $X^0:=(W^0)^n\cap X$.
\item $W^0_\Gg :=\Gg (W^0), \Gg \in \GG$.
\item Let $\nu :V \to k^n$ be the map $\nu (w_1,\dots ,w_n)=(\mu (w_1),\dots , \mu (w_n))$. It is clear that the restriction of $\nu$ on $X$ defines a map $\nu :X\to L$.
\end{enumerate}
\end{definition}

We start with some observations:
\begin{claim}\label{tr}
\begin{enumerate}
\item For any  $x\in X^0, 1\leq i\leq n$ there exist unique 
$t(x)\in T$ such that $x=t(x)\kk (\nu (x))$.
\item $f(x)=\theta (t(x))f(\kk (\nu (x)))$ for any $x\in X^0$.
\item For any $\Gg \in \GG$, $l\in L$ we have $\nu (\Gg (l))=l$.
\end{enumerate}
\end{claim}

\begin{lemma}\label{zero} We have $\mcP_a^ {\bar w}(X)^\theta =\{0\}$ for any 
$\theta \not \in \Theta^{adm}$.
\end{lemma}

\begin{proof}Assume that if 
$\theta \not \in \Theta^{adm}$. Then there exist $i,j,j',1\leq i\leq n, 1\leq  j,j'\leq d$ such that $\alpha _{j,j'}(\chi _i)\geq a$. Choose $s\in S_d$ be such $s(j)=1,s(j')=2$ and denote by $\ti s\in \GG$ the image of $s$ under the imbedding $S_d\ho \GG$ as the $i$-factor. After the replacement $f\to f\circ \ti s, \theta \to \theta \circ s$ we may assume that  $\alpha _{1,2}(\chi _i) > a$.

The functions $h:=\kk ^\star (f)$ and  $h_s:=\kk _s^\star (f)$ are weakly polynomial functions of degrees $\leq a$ on the linear space $L$. Therefore $h$ and $h_s$ are polynomial functions of degrees $\leq a$.

Let $\phi :=\zeta _i\circ \phi _{1,2}:\D  \to T$. Then $\kk _s(l)=\phi (l_i)\kk (l)$, $l=(l_1, \dots ,l_n)\in L$. Therefore for any $l\in L$ such that $l_i\in \D$ we have
$h_s(l)=l_i^{\alpha _{1,2}(\chi _i)}h(l)$. Since $\alpha _{1,2}(\chi _i)>a$ and $\alpha _{1,2}(\chi _i)\le m/2$ this is only possible if $h=0$. 
\end{proof}

\begin{corollary} As follows from
 Claim \ref{plus} it is sufficient to prove Proposition \ref{Id} for 
$\theta \in \Theta ^{adm,+}$. 
\end{corollary}

Since $L\subset k^n$ is a linear subspace we can choose a polynomial $F_h$ on $V $of degree $\leq a$ extending $h.$ One way to define $F_h$ a the composition $h\circ \pi$ where 
$\pi :k^n\to L$ is a a linear projection.
\begin{remark}Since the map $\mcP _a(X)^\theta \to \mcP _a(L), F\to F_{|L} $ an a choice of an extension of $f$ to $F\in \mcP _a(X)^\theta $ is determined by a choice of an extension of $h$ to a polynomial $F_h$ on $k^n$.
\end{remark}
\begin{definition}  

$P:V\to k$ be the polynomial given by
$$P(v)=\prod _{i=1}^n \prod _{j=1}^d(x_i^j)^{\alpha ^{1,j}(\chi _i)}F_h(\nu (v)), \ v=( x_i^j)$$
\end{definition}
\begin{lemma}$deg(P)\leq ad$.
\end{lemma}
\begin{proof}Let $b=deg(h)$. It is sufficiently to show that for any sequence $\bar e=\{ e(i)\}, e(i)\in [1,d],1\leq i\leq n$ we have 
$\sum _{i=1}^n \alpha ^{1,e(i)}(\chi _i) +b\leq a$.

Suppose there exists $\bar e$ such $\Gg \in \GG$ such that $\sum _{i=1}^n \alpha ^{1,e(i)}(\chi _i) +b>a$.
Since $\theta \in \Theta _n^{adm}$ there exists a subset $I$ of $[1,n]$ such that 
$a<\sum _{i\in I} \alpha ^{1,e(i)}(\chi _i) +b\leq 2a$.

Let $\Gg \in \GG \Gg =\{\bs _i\}, \bs _i \in S_d$ such that  
$e(i)=\bs _i(1)$ for $i\in I$ and $\bs _i=Id$ if $i\not \in I$. Consider $h_{\Gg}:= \kk _\Gg ^\star (f)$. On one hand it is a polynomial of degree $\leq a$ on $L$. On the other $h_\Gg (l)=h(l)\prod _{i\in I}l_i^{ \alpha ^{1,e(i)}(\chi _i) }$. The inequalities 
$a<\sum _{i\in I} \alpha ^{1,e(i)}(\chi _i) +b\leq 2a$
imply that $h\equiv 0$.
\end{proof}

By construction $P_{|L}\equiv f _{|L}$.
Let $\bar f:=f-P$. Then $\bar f$ is weakly polynomial function of degree $\leq ad$ vanishing on $L$ and $\bar f(tx)=\theta (t)f(x)$ for $t\in T, \ x\in X$. By Claim \ref{tr} we see that $\bar f_{|X^0}\equiv 0$. 
We will show that $\bar f\equiv 0$. We start with the following observation.
\begin{remark} 
Since $\bar f$ is a weakly polynomial function  of degree $\le ad$ and $q>ad$ we know that $f{|\mcL}\equiv 0$ for any line $\mcL \subset X$ such that $f$ vanishes on more then $ad$ points on $\mcL$.
\end{remark}

Let $Y:=\bigcup _{\Gg \in \GG}\Gg (X_0)$.
\begin{lemma}\label{van}$\bar f_{|Y} \equiv 0 $
\end{lemma} 

\begin{proof} Since $\Gg (X^0) =TL_\Gg$ it is sufficient to show that
 $f _{|L_\Gg} \equiv 0$ 
for all $\Gg \in \GG $. Let $h_\Gg :L\to k$ be given by 
$h_\Gg (l)=\bar f(\Gg (l))$. We have to show that $h_\Gg \equiv 0$. 
Since 
$h_\Gg$ is a polynomial of degree $\leq ad$ it follows from Claim 
\ref{many} that it is sufficient to show that 
the restriction of $h_\Gg$ on $L(m)$ vanishes. But for any $l\in L(m)$ we
 have  $\Gg (l)=tl', t\in T_n,l'\in L$. Since $\bar f(tx)=\theta (t)f(x)$
we see that $h_\Gg (l)=0$.
\end{proof}

Let $W_1\subset W$ be the subset of  $(x_1,\dots ,x_d)$ for which there exists $j_0, 1\leq j_1\leq d$ such that $x_j\in \D$, for $j\neq j_1$.

Let $W_2\subset W$ be the subset of  $(x_1,\dots ,x_d)$ for which there exists $j_1\neq j_2, 1\leq j_1, j_2\leq d$ such that $x_j\in \D$, for $j\neq j_1, j_2$. Define $Y^1_2=W_2\times (W_1)^{n-1}$. 

\begin{claim} $\bar f_{|Y^1_2}\equiv 0$.
\end{claim}
\begin{proof} Fix $x=(x_i^j) \in Y^1_2$. By definition after a replacement of $x$ by $\Gg tx, t\in T,\Gg \in \GG$ we can assume that $x_1^j=1$ for $j\geq 3$, and $x_i^j=1$  for $j\geq 2$.

Consider the subset 
$$M=\{(y_i^j) \in X |y_1^2=x_1^2, y_1^j=1 ,j\geq 3, y_i^j=1,j\geq 2\}$$

A point in $M$ corresponds to $\{ y_i^1\}_{i=1}^n$ such that $y_1^1x_1^2+\sum _{i=2}^ny_i^1=0$. So $M$ is a linear subspace in $X$. Let 
$N :=\{ y_i^j \in M| y_1^1\in \D \}$. 

Consider the restriction 
$h_M$ of $\bar f$ on $M$. Since $\bar f$ is a weakly polynomial function of degree $\leq ad$ and $M$ is vector space we see that $h_M$ is a  polynomial function of degree $\leq ad$ on $M$. Since $N\subset Y$ we see that ${h_M}_{|N}=0$. It follows now from Claim \ref{many} that $h_M=0$.
So $\bar f(x)=h_M(x)=0.$
\end{proof}
Repeating the same arguments we see that $\bar f=0$.\\

We proved the first part Proposition \ref{Id}.   The second part follows from the following general result.
\begin{lemma}\label{almost}
 Let $\mZ\subset \mV$ be a homogeneous $k$-subvariety of degree   $d,f:Z\to k$ a polynomial function of degree $ad$ which is a weakly polynomial function on $Z$ of degree $\leq a$. Then it is a restriction of polynomial function on $V(k)$ of degree $\leq a$. \end{lemma}
 
 Lemma \ref{almost} follows inductively from the following claim

\begin{claim}\label{we} Let $f:Z\to k$ be a polynomial function of degree $\leq a$ which is weakly polynomial of degree $<a$ and $a<q$. Then $f$ is polynomial of degree $<a$.
\end{claim}

\begin{proof} 
We can write $f$ as a sum $f=Q+f'$ where $deg(f')<a$ and $Q$ is homogeneous of degree $a$.
Since $f$ is weakly polynomial of degree $< a$ the function 
$h$ is also weakly polynomial of degree $<a$. It is sufficient to show that $h\equiv 0$.

Choose $x\in X$ and consider the function $g$ on $k,g(t)=h(tx)$. Since $X$ is homogeneous  $tx \in X$. Since $Q$ is homogeneous of degree $a$ we have $g(t)=ct^a$. On the other hand, since $Q$ is  weakly polynomial of degree $<a$ we see that $g$ is a polynomial of degree $<a $. Since $a<q$ we see that $g\equiv 0$. So $Q(x)=g(1)=0$.
\end{proof}

This completes the proof of Theorem \ref{equality}.
\end{proof}

\section{Proof of Theorem \ref{Jan}.}\label{jan-finite}
We use notation of Definition \ref{proj}. So 
for any $k$-vector field $\mV ,P\in \mcP _d(\mV)$ we consider the map 
$\kk _P: Hom_{af}(W,V) \to \mcP _d(W)$. The first part of following result implies that finite fields $k$ of characteristic $>d$ have the property $c_d$.

\begin{proposition}\label{Jan1}
For any $m\geq 0$ there exists $r=r(m,d)$ such that the following holds. For any finite field $k$ of characteristic $>d$, any polynomial  $P\in \mcP _d(\mV) $ of rank $\geq r$ and 
any  $R\in \mcP _d(\mA ^m) $ the fiber
 $\mZ_Q:=\kk _P^{-1}$ is {\it large}. That is 
\begin{enumerate}
\item $Z_Q(k)\neq \emp$.
\item Let $k_l/k$ be the extension of degree $l$, then
$$lim _{l\to \infty}\frac {log _q (| \kk _P^{-1}(Q)(k_l) |)}{l}= dim(Hom_{af}(W,V))- dim(\mcP _d(W)).$$ 
\end{enumerate}
\end{proposition}

\begin{remark} 
\begin{enumerate}
\item It is clear that Proposition shows that $k$ satisfies the first condition of the property $c_d$. We show in Appendix C that it also implies the $k$ satisfies the second condition.
\item A similar proof to the one below gives the more general result: For any $m\geq 0$ there exists $r=r(m,d,d')$ such that for any finite field $k$ of characteristic $>dd'$, any polynomial 
$P\in k[V_n]$ of degree $d$ and  of $d$-rank $\geq r$ and 
any polynomial $R\in k[x_1,\dots ,x_m]$ of degree $\leq dd'$ there exists a degree $d'$ map 
$\phi :\mA ^m\to V_n$ such that $R=\phi^\star (P)$.
\end{enumerate}
\end{remark}

\begin{proof}
Let $P:V=k^n \to k$ be a polynomial.  Denote by $W$ the space linear maps $w:k^m\to k^n$.
Consider the polynomial $Q$ on $W\times k^{m}$ ded by
$$Q(w,x)=P(w(x))=\sum _{\bl \in \bL}c_\bl (w)x^{\bl},
$$
where $\bL$ is the set of ordered tuples $(j_1, \ldots, j_m)$ with $j_i \ge 0$ and  $\sum_{i=1}^m j_i \le d$.

\begin{lemma} If $p>d$ and  $rank(P)>r$ then $\{c_{\lambda}(w)\}_{\bl \in \bL}$ is of rank $\ge r$.
\end{lemma}

\begin{proof}
We begin with the argument for the case $d=2$.
We are given $P(t)=\sum_{1 \le  i \le j \le n}a_{ij}t_it_j +\sum_{1 \le  i \le n}a_{i}t_i +a$ of rank $r$.
Note that  for any linear form  $l(t)=\sum_{i=1}^n c_it_i$ we have that $P(t)+l(t)$ is of rank $\ge r$.

We can write
\[\begin{aligned}
P(w(x)) &= \sum_{1 \le i \le j \le n} a_{ij}(w^i(x)+s_i)(w^j(x)+s_j)+ \sum_{1 \le i \le n} a_{i}(w^i(x)+s_i)+a \\
&= \sum_{1 \le i \le j \le n} a_{ij}\sum_{k,l=1}^m w^i_{k}w^j_{l}x_kx_l+\sum_{1 \le i \le n} a_{i}\sum_{k=1}^m w^i_{k}x_k \\
&+  \sum_{1 \le i \le j \le n} \sum_{k=1}^m a_{ij}(s_iw^j_k+s_jw^i_k)x_k+ \sum_{1 \le i \le j \le n} a_{ij}s_is_j + \sum_{1 \le i \le n} a_{i}s_i +a
\end{aligned}\]
Which we can write as
\[\begin{aligned}
&\sum_{1 \le k < l \le m}   \sum_{1 \le i \le j \le n} a_{ij}(w^i_{k}w^j_{l} + w^i_{l}w^j_{k}) x_kx_l +
\sum_{1 \le  l \le m} \sum_{1 \le i \le j \le n} a_{ij}w^i_{l}w^j_{l}x^2_l \\
&  + \sum_{k=1}^m [\sum_{1 \le i \le n} a_{i}\ w^i_{k}
+ \sum_{1 \le i \le j \le n}  a_{ij}(s_iw^j_k+s_jw^i_k)]x_k + \sum_{1 \le i \le j \le n} a_{ij}s_is_j + \sum_{1 \le i \le n} a_{i}s_i +a
\end{aligned}\]
We want to show that  the collection
\[\begin{aligned}
\{& \sum_{1 \le i \le j \le n} a_{ij}(w^i_{k}w^j_{l} + w^i_{l}w^j_{k})+ \sum_{1 \le i \le n} a_{i} w^i_{k}\}_{1 \le k  < l \le m} \bigcup\{ \sum_{1 \le i \le j \le n}  a_{ij}(s_iw^j_k+s_jw^i_k)\}_{1 \le k \le m} 
\\&\bigcup\{\sum_{1 \le i \le j \le n} a_{ij}s_is_j +\sum_{1 \le i \le n} a_{i} s_i\}
 \bigcup
\{ \sum_{1 \le i \le j \le n} a_{ij}w^i_{l}w^j_{l}\}_{1 \le l \le m}
\end{aligned}\]
is of rank $\ge r$. Namely, 
 we need to show that  if $B=(b_{kl}), (c_k), (d_k)$ is not $0$ then
\[\begin{aligned}
&\sum_{1 \le k < l \le m}  b_{kl} [\sum_{1 \le i \le j \le n} a_{ij}(w^i_{k}w^j_{l} + w^i_{l}w^j_{k})]
+\sum_{1 \le l \le m}  b_{ll} [\sum_{1 \le i \le j \le n} a_{ij}w^i_{l}w^j_{l} ]\\
&+ \sum_{1 \le k \le m} c_k [\sum_{1 \le i \le j \le n}  a_{ij}(s_iw^j_k+s_jw^i_k)]
+ \sum_{1 \le k \le m} d_k [\sum_{1 \le i \le j \le n} a_{ij}s_is_j +\sum_{1 \le i \le n} a_{i} s_i]
\end{aligned}\]
is or rank $\ge r$.  Suppose $b_{11} \ne 0$. Then we can write the above as
 \[
 b_{11}P(w_1) + l_{w_2, \ldots, w_m}(w_1)
\]
where $w_j=(w_j^1, \ldots, w_j^n)$, and $l_{w_2, \ldots, w_m}$ is linear in  $w_1$, so as a polynomial in $w_1$ this is of rank $\ge r$ and thus also of rank $\ge r$ as a polynomial in $w$.  Similarly in the case where $b_{ll} \ne 0$, for some $1 \le l \le m$.

Suppose $b_{12} \ne 0$.  We can write the above as
\[
(*) \quad  b_{12}Q(w_1, w_2) + l_{w_3, \ldots, w_m}(w_1, w_2)
\]
where $Q:V^2 \to k$, and $Q(t,t)=2P(t)$, and $l_{w_3, \ldots, w_m}:V^2 \to k$ is linear. Thus restricted to the subspace
in $W$ where $w_1=w_2$ we get that $(*)$ is of rank $\ge r$ and thus of rank $ \ge r$ on $W$.  Similarly if $b_{kl} \ne 0$ for some
$k< l$. Similar analysis for $c_k$ or $d_k$ not zero.  \\

For $d>2$ the argument is similar: We carry it out in the case $P$ is homogeneous; the non homogeneous case is similiar. We are given $P(t)=\sum_{I \in \mcI}a_{I}t_I$  of rank $r$, where $\mcI$ is the set of
ordered tuples $(i_1, \ldots, i_d)$ with  $1 \le i_1 \le \ldots \le i_d \le n$, and $t_I= t_{i_1} \ldots t_{i_d}$.

Note that  for any polynomial  $R(t)$ of degree $<d$  we have that $P(t)+R(t)$ is also of rank $>r$.

We can write
\[
P(w(x)) = \sum_{I \in \mcI} a_{I}w^I(x)
= \sum_{I \in \mcI} a_{I}\sum_{l_1, \ldots, l_d=1}^m w^{i_1}_{l_1} \ldots w^{i_d}_{l_d}x_{l_1} \ldots x_{l_d}
\]
For $1 \le l_1 \le \ldots \le l_d \le m$ the term  $x_{l_1} \ldots x_{l_d}$ has as coefficient
\[
 \sum_{I \in \mcI_{l_1, \ldots, l_d}} a_{I}w^{i_1}_{l_1} \ldots w^{i_d}_{l_d}
\]
where $\mcI_{l_1, \ldots, l_d}$ is the set of permutations of  $l_1, \ldots, l_d$.  We wish to show that the set of
the collection
\[
\{\sum_{I \in \mcI_{l_1, \ldots, l_d}} a_{I}w^{i_1}_{l_1} \ldots w^{i_d}_{l_d}\}_{ \mcI_{l_1, \ldots, l_d}}
\]
is of rank $\ge r$  Namely we need to show that  if $B=(b_{ \mcI_{l_1, \ldots, l_d}})$ is not $0$ then
\[
\sum_{ \mcI_{l_1, \ldots, l_d}} b_{ \mcI_{l_1, \ldots, l_d}} \sum_{I \in \mcI_{l_1, \ldots, l_d}} a_{I}w^{i_1}_{l_1} \ldots w^{i_d}_{l_d}
\]
is or rank $\ge r$.  Suppose $b_{ \mcI_{l_1, \ldots, l_d}} \ne 0$. Then restricted to the subspace
$w_{l_1} = \ldots = w_{l_d}$ we can write the above as
 \[
 b_{ \mcI_{l_1, \ldots, l_d}} |\mcI_{l_1, \ldots, l_d}|P(w_{l_1}) + R(w)
\]
where $w_j=(w_j^1, \ldots, w_j^n)$, and $R(w)$ is of lower degree in  $w_{l_1}$, so as a polynomial in $w_{l_1}$ this is of rank $\ge r$ and thus also of rank $ \ge r$ as a polynomial in $w$.
\end{proof}

\begin{lemma}\label{onto} There exists $r=r(d,m)$ such that for $P:k^n \to k$ is of rank $>r$ the map  $f:W\to k^{|\bL|}$  given by
$w\to \{ c_\bl (w)\}_{\bl \in \bL}$ is surjective.
\end{lemma}
\begin{proof}
We need to solve the system of equations $ c_\bl (w) = b_{\bl}$ for any $b \in k^{\bL}$.  The number of solutions is given by
\[
q^{-|\bL|}\sum_{\alpha \in k^{\bL}} e_q(\sum_{\bl \in \bL} a_{\bl}(c_\bl (w) -b_{\bl})).
\]
Let $s=2|\bL| \le d^m$. By Proposition \ref{bias-rank-1} there exists $r=r(d,m)$ such that each term other that the term corresponding to $\alpha =0$ is
at most $q^{-s}$, so that the total contribution from these terms is at most  $q^{-|\bL|}$.
\end{proof}

This completes the proof of Proposition \ref{Jan1} and the first part of Theorem \ref{Jan}. We show in Appendix $C$ how to deduce the  second part of Theorem \ref{Jan} from the  second part of Proposition \ref{Jan1}.

\end{proof}


\section{Proof of Proposition \ref{ext}}\label{jan-extension}

A key tool in our  proof of this Proposition is a testing result from \cite{kz-uniform} which roughly says that any weakly polynomial function of degree $a$ that is "almost" weakly polynomial of degree $<a$, namely it is a polynomial of degree $<a$ on almost all affine subspaces, is weakly polynomial of degree $<a$.  This does not require high rank. We use high rank to show that almost any isotropic line is contained in an isotropic plane that not contained in  $l^{-1}\{0\}$. \\

\begin{proof}[Proof of Claim \ref{ext}]
Let $V$ be a vector space and $l:V\to k$ be a non-constant affine function. For any subset $I$ of $k$ we denote $\mW _I=\{ v \in \mV |l(v)\in I\}$ so that  $\mW _b =\mW _{\{ b\} }$, for $b\in k$.
For a hypersurface $\mX \in \mV$ we write $\mX _I=\mX \cap \mW_I$.

\begin{lemma}\label{l} For any  finite subset $S\subset k$, any a weakly polynomial function $f$ of degree $a$ on $X$ such that 
$f_{|X_S} \equiv 0$, and any $b\in k$ there exists  a polynomial $Q$ of degree $\leq a$ on $V$ such that 
$Q_{|X_S}\equiv 0$ and
 $(Q-f)_{|X_b} \equiv 0$.
\end{lemma}

\begin{proof}
We start with the following result.

\begin{claim}Under the assumptions of Lemma \ref{l} the restriction $f_{|X_b}$ is a weakly polynomial function of degree $\leq a-|S|$.
\end{claim} 

\begin{proof} Since $|k|>a$ it suffices to show that  for any plane $L\subset X_b$ the restriction $f_{|L}$ is a polynomial of degree  $\leq a-|S|$.

As follows from Proposition \ref{testing-lines}, there is a constant $A= A(d,a)$ such that it suffices to check the restriction $f_L$  on $q^{-A}$-almost any affine plane $L\subset X_b$ is a polynomial of degree $\leq a-1$.

As follows from Proposition \ref{line-plane} for any $s>0$ there is an $r=r(d,s)$  such that if $X$ is of rank $>r$ then for $q^{-s}$-almost any  affine plane $L\subset X_b$ there exists  an affine $3$-dim subspace  $M\subset X$ containing $L$ and such that $M\cap W_0\neq \emp$.  Then $M\cap X_t\neq \emp$ for any $t \in k$. Since $f$ is a weakly polynomial function of degree $a$ its restriction to $M$ is a polynomial $R$ of degree $\leq a$. Since the restriction of $R$ to  $l^{-1}(S) \cap M \equiv 0$ we see that $R=R'\prod _{s\in S}(l-s).$ 
Since $l{|L}\equiv b$ we see that the restriction $f_L$  is equal to $R'$ which is a 
is a polynomial of degree $\leq a-|S|$.
\end{proof}
Now we show that this Claim implies Lemma \ref{l}. Indeed, assume that $f_{|X_b}$ is a weakly polynomial function of degree $\le a-|S|$.  It follows from the inductive assumption  on $a$ that there exists  a polynomial $Q'$ of degree $\leq a-|S|$ on $V$ such that $f_{|X_b}=Q'_{|X_b}.$  Let $Q:=\frac {Q'\prod _{s\in S}(l-s)}{\prod _{s\in S}(b-s) }$. Then $(f-Q)_{X_{S \cup \{b\}}}\equiv 0.$
\end{proof}
This completes the proof the Claim \ref{ext}.
\end{proof}
\ \\

To formulate a useful variant of Corollary \ref{extension1} we introduce some notations. Recall the definitions of 
$\mcP _a^w(X)$ , $\mcP _a(X)$ from Definition \ref{X} 
\begin{definition}\label{weak} 
\begin{enumerate}
\item  We denote by $\overline \mcP _a(X)$ the quotient $\mcP _a^w(X) /\mcP _a(X)$.
\item For an imbedding $\phi :V'\ho V$ we define $X_\phi :=\{ v'\in V'|\phi (v')\in X\}$. The map $\phi$ defines an imbedding $X_\phi \to X$ which we also denote by $\phi$.
\item We denote by  $\phi ^\star : k[X] \to k[X_\phi] $ the map $f\to f\circ \phi$.
\item We denote by $$\bar \phi : \overline \mcP _a(X) \to \overline \mcP _a(X_{\phi}) $$ the linear map induced by $\phi^\star$.
\end{enumerate}
\end{definition}

\begin{corollary} \label{extension2} 
 Let $a>0$, and let $k$ be a field with $a<q$. There exists $\rho =\rho (d,a)$ such that 
for any hypersurface $X= \{v\in V|P(v)=0\}\subset V$ of degree  $\leq d$, any affine imbedding 
$\phi :W\ho V$ such that rank of $P_{|W}\geq \rho$ the restriction map 
$\bar \phi : \overline \mcP _a(X) \to \overline \mcP _a(X\cap W) $ is an imbedding.
\end{corollary}


\section{Complete intersections of bounded codimension}\label{subvariety}
The arguments in the paper are written in the case when $\mX$ is hypersurface. Most of our results extend easily to the case when  $\mX$ is a complete intersection of bounded codimension.

\begin{definition}[Algebraic rank of a variety]\label{def-rank-variety} 
\begin{enumerate}
\item Given a collection $\bar P \in  \mcP _{\bar d}(\mV) $  of polynomials we define 
$\mX _{\bar P}=\{ v\in \mV|P_s(v)=0, 1\leq s\leq c\}$.
\item A collection $\bar P=\{P_i\}$ of polynomials is {\it admissible} if 
$\mX _{\bar P}\subset \mV $ is a complete intersection (that is $\dim(\mX _{\bar P})=\dim(\mV)-c$).
\item In the case when all polynomials $P_i$ are of the same degree we define 
$r(\bar P)$ as  the minimal rank of a non trivial $k$-linear combination of $P_s$.
\item  If $\bar P=\bigcup _j\bar P_j$ where $P_j$ is a collection of polynomials of degree $d_j$ we 
say that $r(X_{\bar P})>r$ if  for all $j$, $r(\bar P_j)>r$.
\end{enumerate}
\end{definition}

Let $P_{n,d}: \mV _n\to \mA$ be polynomials as in Theorem \ref{const} and $\mX _{n,d}=\mX _{P_{n,d}}$. For any $\bar d=(d_1,\dots ,d_c)$ we define $\bar P_{n,\bar d}\in \mcP _{\bar d}(\mV ^c)=\{ P_{l,d_l\circ p_i}\}$, $1\leq l\leq c$ where $p_l:\mV _n^c\to \mV$ is the projection on the $l$-component. It is clear that $\mX _{\bar P}=\prod _{l=1}^c\mX _{n,d_l}$.

\begin{theorem}\label{co}
\begin{enumerate}
\item $r(\bar P_{n,\bar d} )\geq n/d$, where $d:=\max _ld_l$.
\item For any admissible field $k$,  the complete intersection   $\mX _{n,\bar d}\subset \mV _n^c$ has the property  $\star _a$.
\end{enumerate}
\end{theorem}
\begin{proof}The proof of the  part $(1)$ is completely analogous to the 
proof of the  part $(1)$ of Theorem \ref{const}.
The proof of the second part is also parallel to proof of the  part $(2)$ of Theorem \ref{const} which is based on 
on the analysis of the restriction of weakly polynomial functions to the subspace $L\subset \mX _n$
and the decomposition of the space $k[\mX _n(k)]$ under the action of the torus $T$. In the proof of Theorem \ref{co} we use the restriction weakly polynomial functions to the subspace $L^c\subset \mX _{n,\bar d}$ and the action of the torus $T^c$ on the space $k[\mX _{n,\bar d}(k)]$. 
\end{proof}

One checks that the formulations and proofs of Claim \ref{planes} and of Proposition \ref{ext} are naturally extended to complete intersections of bounded codimension.

The following is the counterpart of Proposition \ref{Jan1}.
We can consider $\mcP _{\bar d}$ as a contravariant autofunctor on the category Vect$_{af}$ of finite-dimensional $k$-affine vector spaces where to an affine map $f:W\to V$ we associate the map 
$$f^\star :\mcP _{\bar d}(V) \to \mcP _{\bar d}(W), \quad  f^\star (\{P_s\})= \{P_s \circ f\}.$$

\begin{proposition}\label{Jan2}For any $m\geq 1$ There exists $r=r(\bar d,m)$ such that for any $\bar P \in \mcP _{\bar d}(V)$, $\bar Q\in \mcP _{\bar d}(U)$, with $r(\bar P)\geq r$, and dim$(U)\leq m  $ there exists $f\in \text{Hom}_{\text{Vect}_{af}}(U,V)$ such that $\bar Q=f^\star (\bar P)$.
\end{proposition}

\begin{proof}
Consider the polynomials $R_s, \ 1 \le s \le c$ on $W\times k^{m}$ defined by

$$R_s(w,x)=P_s(w(x))=\sum _{\bl \in \bL_s}c^s_\bl (w)x^{\bl}, $$
where $\bL_s$ is the set of ordered tuples $(j_1, \ldots, j_m)$ with $j_i \ge 0$ and  $\sum_{i=1}^m j_i \le d_s$.

\begin{lemma}\label{C} If $p>d$ and   $rank(\bar P)>r$ then $\{c^s_{\lambda}(w)\}_{1 \le s \le L,\bl \in \bL}$ is of rank $>r$.
\end{lemma}

\begin{proof}
We are given $P_s(t)=\sum_{I \in \mcI_s}a^s_{I}t_I$, $1 \le s \le c$,  of rank $r$, where $\mcI_s$ is the set of
ordered tuples $(i_1, \ldots, i_{d_s})$ with  $1 \le i_1 \le \ldots \le i_{d_s} \le n$, and $t_I= t_{i_1} \ldots t_{i_{d_s}}$.

Note that  for any polynomials $l_s(t)$ of degrees $<d_s$  we have that $\{P_s(t)+l_s(t)\}$ is also of rank $>r$.

We can write
\[
P_s(w(x)) = \sum_{I \in \mcI_s} a^s_{I}w^I(x)
= \sum_{I \in \mcI} a^s_{I}\sum_{l_1, \ldots, l_{d_s}=1}^m w^{i_1}_{l_1} \ldots w^{i_{d_s}}_{l_{d_s}d}x_{l_1} \ldots x_{l_{d_s}}
\]
For $1 \le l_1 \le \ldots \le l_{d_s} \le m$ the term  $x_{l_1} \ldots x_{l_{d_s}}$ has as coefficient
\[
 \sum_{I \in \mcI_{l_1, \ldots, l_{d_s}}} a^s_{I}w^{i_1}_{l_1} \ldots w^{i_{d_s}}_{l_{d_s}},
\]
where $\mcI_{l_1, \ldots, l_{d_s}}$ is the set of permutations of  $l_1, \ldots, l_{d_s}$.  We wish to show that 
the collection
\[
\{\sum_{I \in \mcI_{l_1, \ldots, l_{d_s}}} a^s_{I}w^{i_1}_{l_1} \ldots w^{i_{d_s}}_{l_{d_s}}\}_{ 1 \le s \le c, \mcI_{l_1, \ldots, l_{d_s}}}
\]
is of rank $>r$.  Write $[1,c]=\bigcup_{f=2}^d C_f$ where $C_f=\{s: d_s=f\}$. 

We need to show that for any $f=2, \ldots, c$ if $B=(b_{ \mcI_{l_1, \ldots, l_{d_s}}})_{s \in C_f,   \mcI_{l_1, \ldots, l_{d_s}}}$ is not $0$, then
\[
\sum_{s\in C_f} \sum_{ \mcI_{l_1, \ldots, l_{d_s}}} b_{ \mcI_{l_1, \ldots, l_{d_s}}} \sum_{I \in \mcI_{l_1, \ldots, l_{d_s}}} a^s_{I}w^{i_1}_{l_1} \ldots w^{i_{d_s}}_{l_{d_s}}
\]
is or rank $>r$.  Suppose $(b_{ \mcI_{l_1, \ldots, l_{d_s}}})_{s\in C_f} \ne \bar 0$. Then restricted to the subspace
$w_{l_1} = \ldots = w_{l_{d_s}}$ we can write the above as
 \[
\sum_{s \in C_f} b_{ \mcI_{l_1, \ldots, l_{d_s}}} |\mcI_{l_1, \ldots, l_{d_s}}|P_s(w_{l_1}) + R(w)
\]
where $w_j=(w_j^1, \ldots, w_j^n)$, and $R(w)$ is of lower degree in  $w_{l_1}$, so as a polynomial in $w_{l_1}$ this is of rank $>r$ and thus also of rank $>r$ as a polynomial in $w$.
\end{proof}

\begin{lemma} There exists $r=r(d,c,m)$ such that for $\bar P:k^n \to k^c$ is of rank $>r$ the map  $f:W\to \prod_s k^{|\bL_s|}$  given by $w\to \{ c_\bl (w)\}_{s, \bl \in \bL_s}$ is surjective.
\end{lemma}

\begin{proof}
Same proof as the corresponding Lemma \ref{onto} for hypersurfaces. 
\end{proof}

This completes the proof of Lemma \ref{C} and therefore Proposition \ref{Jan2}.
\end{proof}

Finally, Proposition \ref{ext} for complete intersections is proved in the exact same way.

  
\appendix

\section{Counting tools}
The main property of high rank varieties is that it is easy to estimate the number of points on various important varieties. The main counting tool comes from the relation between the bias of exponential sums and algebraic rank.

Let $k$ be a finite field, char$(k)=p$, $|k|=q$. Let  $V$ a vector space over $k$. We denote $e_q(x) = e^{2 \pi i \psi(x)/p}$ where $\psi:k \to \mathbb F_p$ is the trace function. 
Let $P :V \to k$ be a polynomial of degree $d$. We denote by $(h_1, \ldots, h_d)_P$ the multilinear form
 \[
(h_1, \ldots, h_d)_P=\sum_{\bo \in \{0,1\}^d} -1^{|\bo|}P(x+\omega \cdot \bar h); \qquad |\bo| = \sum_{i=1^d} \bo_i.
\]

We denote by $\mE_{x \in S} f(x)$ the average  $|S|^{-1}\sum_{x \in S} f(x)$. 

\begin{definition}[Gowers norms \cite{gowers}]\label{uniform} For a function $g: V \to \mC$ we define the norm $\|g\|_{U_d} $ by \[\|g\|^{2^d}_{U_d} = \mE_{x,v_1, \ldots v_d\in V} \prod_{\omega \in \{0,1\}^d} g^{\omega}(x+\omega \cdot \bar v),\] where $g^{\omega}=g$  if $|\omega|$ is even and $g^{\bo}=\bar g$ otherwise.
 \end{definition}

\begin{definition}[Analytic rank] The analytic rank of a polynomial $P:V \to k$ of degree $d$ is defined by 
arank$(P)=-\log_q \|e_q(P)\|_{U_d}$. 
 \end{definition}
 
The following Proposition relating bias and rank was proved in increasing generality in \cite{gt1,kl,bl}.
 The most general version can be found at the survey \cite{hhl} (Theorem 8.0.1):
 
 \begin{proposition}[Bias-rank]\label{bias-rank-1}
 Let $s,d>0$. There exists $r=r(s,k,d)$ such that for any finite field $k$ of size $q=p^l$, any vector space $V$ over $k$, any polynomial  $P:V \to k$ of degree $d$. If $P$ is of rank $>r$ then
\[
 |\mE_{v \in V}e_q(P(v))| < q^{-s}.
\]
In the case when $p>d$, the bound on  $r$ is uniform in $k$.
\end{proposition}
\begin{conjecture}\label{conj-bias} For $p >d$ we have $r = s^{-O_d(1)}$. The conjecture is known for $d=2,3$ (\cite{s-h}).
\end{conjecture}

\begin{remark}\label{norm-bias-rank}
When $p>d$ then one can recover $P$ from $(h_1, \ldots, h_d)_P$ so that if the rank of $(h_1, \ldots, h_d)_P$  is  $<r$ of then
so is the rank of $P$. 
\end{remark} 
For multilinear functions; in particular for  $(h_1, \ldots, h_d)_P$, the converse is also true:

\begin{proposition}[\cite{kz-approx}]\label{equi-multi}
 Let $r,d>0$. For any finite field $k$ of size $q=p^l$, $p>d$, any vector space $V$ over $k$, any polynomial  $P:V \to k$ of degree $d$, if
\[
\|e_q(P)\|^{2^d}_{U_d} = |\mE_{h_1, \ldots, h_d} e_q(h_1, \ldots, h_d)_P| <q^{-r}
\]
for some polynomial,
then $P$ is of rank $>r$.
\end{proposition}

\begin{remark}If $P$ is of degree $d$ and $p>d$ then $(h, \ldots, h)_P = P(h)/d!$ so that if $P$ is of rank $>r$ then also $(h_1, \ldots, h_d)_P$ is of rank $>r$ as a polynomial on $V^d$.
\end{remark}

\begin{lemma}\label{bias-rank-2}  For any $R$ of degree $<d$ we have
\[
|\mE_{h_1, \ldots, h_d}  e_q((h_1, \ldots, h_d)_P+R(h_1, \ldots, h_d))| \le |\mE_{h_1, \ldots, h_d}  e_q((h_1, \ldots, h_d)_P)|
\]
\end{lemma}

\begin{lemma}\label{subspace-rank} Let $P:V\to k$ be a polynomial of degree $d$ and rank $R$, and let $W\subset V$ be a subspace of codimension $s$. Then the rank of $P_{|W}$ is $\geq R-s.$
\end{lemma}

\begin{lemma}\label{size}  Let $s>0$, $\bar d = (d_1, \ldots, d_c)$, $k$ a finite field.  There exists $r=r(\bar d, s,k)$ such that for any $\bar P=\{P_1, \ldots, P_c\}$,  $P_i:V\to k$ with deg$(P_i)\le d_i$,  $|X_{\bar P}|=q^{dim(V)-c}(1+q^{-s})$. In the case when $p>\max_i d_i$, the bound on  $r$ is uniform in $k$.
\end{lemma}

\begin{proof}
The number of points on $X$ is given by 
\[
q^{-c}\sum_{\bar a \in k^c}\sum_{x \in V} e_q(\sum_{i=1}^c a_iP_i(x)) .
\]
By Proposition \ref{bias-rank-1} for any $s>0$ we can choose $r$ so that for any $\bar a \ne 0$ we have 
\[
|\sum_{x \in V} e_q(\sum_{i=1}^c a_iP_i(x)) | < q^{-s}|V|.
\]
\end{proof}


\section{Almost-sure results}

In \cite{KR} (Theorem 1)  the following description of degree $<m$ polynomials is given:
\begin{proposition}\label{kau-ron} Let $P:V \to k$. Then $P$ is a polynomial of degree $\le a$ if and only if
the restriction of $P$ to any affine subspace of dimension $l=\lceil \frac{a+1}{q-q/p}\rceil$ is a polynomial of degree $<m$. 
\end{proposition} 

Note that when $a<q$ then $l\le 2$. 

In \cite{KR} the above criterion is used for polynomial testing over general finite fields. In \cite{kz-uniform} (Corollary 1.14) it is shown how the arguments in \cite{KR} can be adapted to polynomial testing within a subvariety variety $X \subset V$ (high rank is not required). 

\begin{theorem}[Subspace splining on $X$]\label{testing-lines} For any  $a,d, L>0$ there exists an $A=A(d,L,a) > 0$ such that the following holds. 
Let $X \subset V(k)$ be  a  complete intersection of degree $d$, codimension $L$. Then any  weakly polynomial  function $f$ of degree $a$
such that the restriction of $f$ to $q^{-A}$-a.e  $l$-dimensional affine subspace, 
$l=\lceil \frac{a}{q-q/p}\rceil$ is a polynomial of degree $<a$ is weakly polynomial of degree $<a$.
\end{theorem}

Let the notation be as in Section \ref{jan-extension}.

\begin{proposition}\label{line-plane} Fix $\bar d =\{ d_i\} $ and $s>0$.
Let $d:=max _id_i $. There exists $r=r(\bar d, s)$ such that for any finite field $k$ with $char(k)>d$, a $k$-vector space $\mV$ and $ \bar P\in \mcP _{\bar d}(\mV)$ of rank $>r$ the following holds. 

 \begin{enumerate}
 \item
For any $b\in k$ and  $q^{-s}$-almost any affine line $L \subset X_b$ there exists an affine plane $M \subset X$ containing $L$ such that $M \cap X_0 \ne \emptyset$.
 \item  For $q^{-s}$-almost any affine plane $M \subset X_b$ there exists an affine $3$-dim subspace $N\subset X$ containing $M$ such that $N \cap X_0 \ne \emptyset$.
 \end{enumerate} 
\end{proposition}

\begin{proof} 
We prove the result for hypersurfaces, the proof for complete intersections in analogous.

We will prove $(1)$; the proof of $(2)$ is similar. 

Let $k=\mF _q$. We fix $d$ and define $d'= min(d+1,q)$.
Let $M_0=\{ a_0,\dots ,a_{d}\}\subset  k$ be a subset of $d'$ distinct points.

\begin{claim} Let $Q(x)$ be a polynomial of degree $\leq d$ such that
$Q_{|M_0}\equiv 0$. Then $Q(a)=0$ for all $a\in k$.
\end{claim}

\begin{proof}
If $q\geq d+1$ the Claim follows from the formula for the Vandermonde determinant. On the other hand if $d\geq q$ then there is nothing to prove.
\end{proof}

From now on we assume that $q\geq d+1$. For any $0\leq i\leq d$ we define a subset $S_i$ of $k^2$ by
\[
S_i=\{ (a_i,a_j), 0\leq j\leq i \}
\]
 Let $T=\bigcup _{0\leq i\leq d}S_i$.

\begin{claim} Let $Q(x,y)$ be a polynomial of degree $\leq d$ such that
$Q_{|T}\equiv 0$. Then $Q=0$.
\end{claim}

\begin{proof} We prove the claim by induction in $d$. Let
$Q=\sum _{a,b}q_{a,b} x^ay^b, a+b\leq d$. The restriction of $Q$ on the line $\{ x=0\}$ is equal to
$Q^0(y)=\sum _{b\leq d}q_{0,b}y^b $. Since $Q^0_{|S_d}\equiv 0$ we see that $Q^0=0$. So $Q(x,y)=xQ'(x,y)$. By the inductive assumption we have $Q'=0$
\end{proof}

We will assume from now on that $a_0 =0$. Denote $I(d)=\{(i,j):  1\leq i\leq d,  \ 1\leq j\leq i\}$, and $m=|I(d)|$. 

An affine line in $X_b$ is parametrized as $x+ty$, $t \in k$, with
\[
(*) \quad Q(x+ty)=0, l(x)=b, l(y)=0.
\]
Let $Y$ be the set of $(x,y)$ satisfying $(*)$.

We need to show that almost every $(x, y) \in Y$ we can find $z$ with
\[
(**) \quad Q(x+ty+sz)=0, \  l(z)=-b,  \quad s,t \in k
\]
or  alternatively
\[
Q(x+ty+sz)=0, \  l(x+ty+sz)=(1-s)b,  \quad s,t \in k
\]
We can reduce this system to
\[
Q(x+a_iy+a_jz)=0, \    l(x+a_iy+a_jz)=(1-a_j)b, \quad  (i,j) \in I(d)
\]

Fix $(x,y) \in Y$ and estimate the number of solutions:
\[
(*) \quad q^{-2m} \sum_{z} \sum_{a,c \in k^{m}  }e_q(\sum a_{ij} Q(x+a_iy+a_jz)+c_{ij} (l(x+a_iy+a_jz)+(a_j-1)b))
\]
Suppose  $a=(a_{ij}) = 0$, but $c=(c_{ij})\ne 0$, and recall that $l(x)=b, l(y)=0$. We have
\[
\sum_ze_q(\sum_{ij} c_{ij} (l(x+a_iy+a_jz)+(a_j-1)b)) = \sum_z e_q(\sum_{ij} c_{ij} (a_jl(z)+a_jb)
\]
Now if $\sum_{ij} c_{ij} a_jl(z) \not \equiv 0$ then the sum is $0$. Otherwise also $\sum_{i,j} c_{ij} a_jb=0$ so that the sum is $|V|$.
\ \\
Now suppose  $a \ne 0$.  Say $a_{i_0j_0} \ne 0$. We estimate
\begin{equation}\label{one}
\mE_{x,y \in V}|\mE_ze_q(\sum a_{ij} Q(x+a_iy+a_jz)+c_{ij} (l(x+a_iy+a_jz)+(a_j-1)b))|^2
\end{equation}
.
\begin{lemma}\label{complexity}
\[
\mE_{x,y,z,z'}\prod _{(i,j) \in I(d)} f_{i,j}(x+a_iy+a_jz) \bar f_{i,j}(x+a_iy+a_jz+a_jz') \le \|f_{i_0. j_0}\|_{U_{d}}
\]
\end{lemma}

\begin{proof}
Without loss of generality $a_1=1$ (make a change of variable $y \to a_1^{-1}y,z \to a_1^{-1}z )$.
We prove this by induction on $d$. When $d=1$ we have $x+y+z, x+y+z'$, and the claim is obvious.  Assume $d>1$. We can write the average as
\[\begin{aligned}
&\mE_{x,y,z,z'} \prod_{(i,j) \in I(d-1)} f_{i,j}(x+a_iy+a_jz) \bar f_{i,j}(x+a_iy+a_jz+a_jz')\\
& \qquad  \qquad \prod_{1\le j \le d} f_{d,j}(x+a_dy+a_jz) \bar f_{d,j}(x+a_dy+a_jz+a_jz'). \\
\end{aligned}\]
Shifting $x$ by $a_dy$ we get
\[\begin{aligned}
&\mE_{x,y,z,z'} \prod_{(i,j) \in I(d-1)} f_{i,j}(x+(a_i-a_d)y+a_jz) \bar f_{i,j}(x+(a_i-a_d)y+a_jz+a_jz')\\
& \qquad  \qquad \prod_{1\le j \le d} f_{d,j}(x+a_jz) \bar f_{d,j}(x+a_jz+a_jz')  \\
\end{aligned}\]
Applying the Cauchi-Schwartz inequality we get
\[\begin{aligned}
&[\mE_{x,y,y',z,z'} \prod_{(i,j) \in I(d-1)} f_{i,j}(x+(a_i-a_d)y+a_jz) \bar f_{i,j}(x+(a_i-a_d)y+a_jz+a_jz')\\
& \qquad \qquad \prod_{(i,j) \in I(d-1)} f_{i,j}(x+(a_i-a_d)y+(a_i-a_d)y'+a_jz)  \\
& \qquad \qquad \qquad \qquad \bar f_{i,j}(x+(a_i-a_d)y+(a_i-a_d)y'+a_jz+a_jz')]^{1/2}\\
\end{aligned}\]
Shifting $x$ by $a_dy$ and rearranging we get
\[\begin{aligned}
&[\mE_{x,y,y',z,z'} \prod_{(i,j) \in I(d-1)} f_{i,j}(x+a_iy+a_jz) \bar  f_{i,j}(x+a_iy+(a_i-a_d)y'+a_jz)\\
&\prod_{(i,j) \in I(d-1)}\bar f_{i,j}(x+a_iy+a_jz+a_jz')  f_{i,j}(x+a_iy+(a_i-a_d)y'+a_jz+a_jz')]^{1/2}\\
\end{aligned}\]

Now if we denote
\[
g_{i,j, y'}(x) =f_{i,j}(x) \bar f_{i,j}(x+y'+(a_i-a_d)y'').
\]
then by the induction hypothesis we get that the above is bounded by
\[
[\mE_{y'}\|f_{i,j}(x) \bar f_{i,j}(x+y'+(a_i-a_d)y'')\|_{U_{d-1}}]^{1/2} \le \|f_{i,j}(x)\|_{U_d}.
\]
for any $(i.j) \in I(d-1)$. \\

We do a similar computation for $(i.j) \in I(d)\setminus \{I(d-1), (d,1)\}$ , splitting
\[\begin{aligned}
&\mE_{x,y,z,z'}\prod_{(i,j) \in I(d-1)} f_{i+1,j+1}(x+a_{i+1}y+a_{j+1}z) \bar f_{i+1,j+1}(x+a_{i+1}y+a_{j+1}z+a_{j+1}z')\\
& \qquad  \qquad \prod_{1\le j \le d} f_{j,1}(x+a_jz) \bar f_{j,1}(x+a_jz+a_jz')  \\
\end{aligned}\]

The only term left uncovered is $f_{d,1}$,  so we split
\[\begin{aligned}
&\mE_{x,y,z,z'} \prod_{(i,j) \in I(d-1)} f_{i+1,j}(x+a_{i+1}y+a_{j}z) \bar f_{i+1,j}(x+a_{i+1}y+a_{j}z+a_{j}z')\\
& \qquad  \qquad \prod_{1\le i \le d} f_{i,i}(x+a_{i}y+a_iz) \bar f_{i, i}(x+a_{i}y+a_iz+a_iz'). \\
\end{aligned}\]
 We make the change of variable $z \to  z-y$ to get
 \[\begin{aligned}
&\mE_{x,y,z,z'} \prod_{(i,j) \in I(d-1)} f_{i+1,j}(x+a_{i+1}y+a_{j}(z-y)) \bar f_{i+1,j}(x+a_{i+1}y+a_{j}(z-y)+a_{j}z')\\
& \qquad  \qquad \prod_{1\le i \le d} f_{i,i}(x+a_iz) \bar f_{i, i}(x+a_iz+a_iz').  \\
\end{aligned}\]
 \end{proof}

By the Lemma \ref{complexity} we obtain that \eqref{one} is bounded by $\|e_q(\sum a_{i_0j_0} Q+c_{i_0j_0} l)\|_{U_d}$.
By Propositions \ref{bias-rank-1},  \ref{bias-rank-2} for there exists $r=r(s)$ such that if  $Q$ is of rank $>r$ then  $\|e_q(\sum a_{i_0j_0} Q+c_{i_0j_0} l)\|_{U_d}<q^{-s}$.  It follows that we can choose $r$ so that for $q^{-s}$ almost all $x,y \in Y$  the contribution to $(*)$ from all  $(a, c)$ with $a \ne 0$ is bounded by $|V|q^{-4m}$.
\end{proof}


\section{The case of algebraically close  fields}\label{algebraic}
We prove Theorem \ref{Jan} for algebraically closed fields. 

We fix $m,d$. We provide a proof for the case when $c=1$; the general case is completely analogous.
 As proven in Section  \ref{jan-finite}, there exists $r=r(m,d)$ such that for any finite field $k$ of characteristic $>d$, a $k$-vector space $\mV$ and a polynomial $P\in \mcP _d(\mV)$ of rank $\geq r$ the map $\kk _P(k)$ is surjective.

We first show the surjectivity of  $\kk _P(k)$ for polynomials $P\in \mcP _d(\mV)(k)$ of rank $\geq r$.

We fix  $\mV =\mA ^n$ and consider $\mcP _d(\mV)$ as a scheme defined over $\mZ$. Let $T$ be the set of sequences $(a_i,b_i),1\leq i\leq r$ such that $0\leq a_i,b_i <d$ and $a_i+b_i \leq d$. For any $t=\{(a_i,b_i)\} \in T$ we denote by $\nu _t:\oplus _{i=1}^r \mcP _{a_i}(\mV) \otimes \mcP _{b_i}(\mV) $ the linear map given by 
$$\nu _i(\{ Q_i\otimes R_i\})=\sum _{i=1}^r Q_i R_i$$
Let $\mY \subset \mcP _d(\mV)$ be constructible subset which is the complement of unions of $\nu _t,t\in T$ and  $\mY (k)\subset \mcP _d(\mV)(k)$ consists polynomials $P$ of rank $>r$
for any algebraically closed field $k$.

Let $\mR \subset \mY$ be the subscheme of polynomials $P$ such that the map $\kk _P$ is not surjective.

\begin{claim}\label{p} $\mR (\bar \mF _p)=\emp$ for any $p>d$.
\end{claim}

\begin{proof} Assume thar 
 $\mR (\bar \mF _p)\neq \emp$.  Then $\mR (k)\neq \emp$ for some finite extension $k$ of $\mF _p$. So there exists a polynomial 
$P\in \mcP _d(k^n)$ of $\bar k $-rank $>r$ such that the map $\kk _P(\bar k)$ is not surjective. So there exist a finite extension $l$ of $k$ and  $Q \in \mcP _d(k^n) $ which is not in the image of $\kk _P(\bar k)$. Then of course it is not in the image of $\kk _P(\bar k)$. On the other hand  $r_l(P)\geq r_{\bar k}(P)$. So by Theorem \ref{Jan} we see that $Q \in Im (\kk _P(l))$.
\end{proof} 

\begin{corollary}
\begin{enumerate}
\item The map $\kk _P(k)$ is surjective for 
any algebraically closed field $k$ of characteristic $>d$ and a polynomial 
$P\in \mcP _d(\mV)$ of ranks $>r$.
\item The map $\kk _P(k)$ is surjective for 
any algebraically closed field $k$ of characteristic $0$ and a 
polynomial 
$P\in \mcP _d(\mV)$ of ranks $>r$.
\end{enumerate}
\end{corollary}

\begin{proof} The part $(1)$ follows from the completeness of the theory $ACF_p$ of algebraically closed fields of a fixed characteristic $p$.  

To prove the part $(2)$ one choses a non-trival ultrafilter $\mcU$ on the set of prime and 
considers the $\mcU$-ultraproduct of theories $ACF_p$. Let $l$ be the $\mcU$-ultraproduct of fields $\bar \mF _p$. As follows Claim \ref{p}  and Theorem of Los the  map $\kk _P(l)$ is surjective for any polynomial $P\in \mcP _d(\mV)$ of ranks $>r$. Since the theory $ACF_p$ of algebraically closed fields of  characteristic $0$ is complete the Corollary is proved.  
\end{proof}

Let $\mT \subset \mY$ be be the subscheme of polynomials $P$ such that there exists $Q\in \mcP _d(\mA ^m)$ such that 
 $dim (\kk _P^{-1}(Q))\neq dim (Hom_{af}(W,V))-dim(\mcP _d(\mW))$. Theorem \ref{Jan}  says that $\mT =\emp$. The same arguments as before show that it is sufficient to prove that  
$$dim (\kk _P^{-1}(Q))= dim (Hom_{af}(W,V))-dim(\mcP _d(\mW))$$ for all finite fields $k=\mF _q$ of charateristic $>d$ and $Q\in \mcP _d(\mA ^m)(k)$. But it  follows from the results of Weil that 
$$dim (\kk _P^{-1}(Q))=lim _{l\to \infty}
\frac {log _q (| \kk _P^{-1}(Q)(k_l) |)}{l}$$ 
where $k_l/k$ is the extension of degree $l$. Now the equality $$dim (\kk _P^{-1}(Q))= \dim (Hom_{af}(W,V))-dim(\mcP _d(\mW))$$ follows from Proposition \ref{Jan1} (2).


\end{document}